\documentclass[10pt]{amsart}

\title{Three-manifolds with constant vector curvature}

\author{Benjamin Schmidt}
\address{Department of Mathematics \\ 
                 Michigan State University \\
                 East Lansing, MI 48824}

\author{Jon Wolfson}

\thanks{The first author was partially supported by NSF grants DMS-0905906 and DMS-1207655.}

\date{\today}

\newtheorem{thm}{Theorem}[section]
\newtheorem{lem}[thm]{Lemma}
\newtheorem{cor}[thm]{Corollary}
\newtheorem{prop}[thm]{Proposition}

\theoremstyle{definition}
\newtheorem{rem}{Remark}[section]

\newtheorem{defn}{Definition}[section]

\numberwithin{equation}{section}

\newcommand{\e}{\varepsilon}
\newcommand{\g}{\gamma}

\renewcommand{\l}{\lambda}

\newcommand{\n}{\nabla}

\newcommand{\Sig}{\Sigma}
\renewcommand{\t}{\tau}

\def\Pb{\ifmmode{\Bbb P}\else{$\Bbb P$}\fi}
\def\Z{\ifmmode{\Bbb Z}\else{$\Bbb Z$}\fi}
\def\Q{\ifmmode{\Bbb Q}\else{$\Bbb Q$}\fi}
\def\C{\ifmmode{\Bbb C}\else{$\Bbb C$}\fi}
\def\R{\ifmmode{\Bbb R}\else{$\Bbb R$}\fi}
\def\H{\ifmmode{\Bbb H}\else{$\Bbb H$}\fi}
\def\S{\ifmmode{S^2}\else{$S^2$}\fi}

\def\tr{\operatorname{tr}}
\def\det{\operatorname{det}}

\def\cvc{\operatorname{cvc}}

\def\sec{\operatorname{sec}}

\def\Ric{\operatorname{Ric}}

\def\S{\mathcal S}
\def\I{\operatorname{Isom}}
\def\di{\operatorname{div}}
\def\vol{\operatorname{vol}}
\def\SO{\operatorname{SO}}
\def\SU{\operatorname{SU}}
\def\SL{\operatorname{SL}}

\usepackage{graphicx} 

\usepackage{amsmath,amsthm,amsfonts,amscd,flafter,epsf}
\usepackage{graphics}
\usepackage{epsfig}
\usepackage{psfrag}

\begin{document}

\maketitle

\begin{abstract}
A connected Riemannian manifold $M$ has {\it constant vector curvature $\e$}, denoted by $\cvc(\e)$, if every tangent vector $v \in TM$ lies in a 2-plane with sectional curvature $\e$.  When the sectional curvatures satisfy an additional bound $\sec \leq \e$ or $\sec \geq \e$, we say that $\e$ is an \textit{extremal} curvature.     

In this paper we study three-manifolds with constant vector curvature.  Our main results show that finite volume $\cvc(\e)$ three-manifolds with extremal curvature $\e$ are locally homogenous when $\e=-1$ and admit a local product decomposition when $\e=0$.  As an application, we deduce a hyperbolic rank-rigidity theorem.  
\end{abstract}

\setcounter{secnumdepth}{1}

\setcounter{section}{0}

\section{\bf Introduction}

 We introduce a new curvature condition called {\it constant vector curvature} and present a number of classification theorems in dimension three. 

A connected Riemannian manifold $M$ has {\it constant vector curvature $\e$}, denoted by $\cvc(\e)$, if every tangent vector $v \in TM$ lies in a 2-plane with sectional curvature $\e$; by scaling the metric on $M$, there is no loss in generality in assuming that  $\e = -1, 0$, or $1$.  When the sectional curvatures satisfy an additional bound $\sec \leq \e$ or $\sec \geq \e$, we say that $\e$ is an \textit{extremal} curvature.

Our definition is partly motivated by a consideration of results on geometric rank-rigidity \cite{ba, busp, con, co, ha, shspwi}.  A manifold $M$ has \textit{positive hyperbolic rank} if along each complete geodesic $\g:\R \rightarrow M$ there exists an orthogonal Jacobi field $J(t)$ with $\sec(\dot{\g},J)(t)\equiv -1$; the notions of \textit{positive Euclidean rank} and \textit{positive spherical rank} are analogously defined.  The condition of constant vector curvature simply replaces the condition ``along each complete geodesic'' with ``at each point''. It replaces a global condition on geodesics with a pointwise condition on vectors. 

It is immediate that surfaces with constant vector curvature have constant sectional curvatures. In Section \ref{section:homogeneous} we classify the complete, simply-connected, and homogeneous $\cvc(\epsilon)$ three-manifolds with extremal curvature $\epsilon$, a class of manifolds properly containing the eight three-dimensional Thurston geometries. In Section \ref{section:cvc(-1)} we prove the following:

\begin{thm}\label{maintheorem-case-1}
Suppose that $M$ is a finite volume $\cvc(-1)$ three-manifold.\\  If $\sec \leq -1$, then $M$ is real hyperbolic.  If $\sec\geq -1$ and $M$ is not real hyperbolic, then its universal covering is isometric to a left-invariant metric on one of the Lie groups $E(1,1)$ or $\widetilde{\SL(2,\mathbb{R})}$ with sectional curvatures having range $[-1,1]$.\\  
\end{thm}

\begin{thm}\label{maintheorem-rr} A finite volume three-manifold $M$ with extremal curvature $-1$ and positive hyperbolic rank is real hyperbolic.
\end{thm}

The left-invariant metrics arising in Theorem \ref{maintheorem-case-1} are classified in Section \ref{section:homogeneous}; the necessity of the finite volume assumption will be addressed in \cite{scwo}.   Alternative proofs of Theorem \ref{maintheorem-rr} (1) appear in \cite{co, ha} when $M$ is compact.  To our knowledge (2) is the first hyperbolic rank-rigidity theorem without the additional assumption of non-positive sectional curvatures.  

In contrast to the $\cvc(-1)$ case, there is flexibility in the construction of $\cvc(0)$ three-manifolds.  Observe that any three-manifold that is locally a Riemannian product of a surface and an interval has $\cvc(0)$ since every tangent vector lies in a tangent plane containing the interval factor, a plane of curvature zero.  The following theorem is a partial converse of this observation. Recall that a point $p \in M$ is said to be \textit{isotropic} if all tangent planes to $p$ have the same sectional curvature.

\begin{thm}\label{maintheorem-case0}
Suppose that $M$ is a complete and finite volume $\cvc(0)$ three-manifold with extremal curvature $0$. The subset of $M$ consisting of non-isotropic points admits a local Riemannian product structure.  In addition, if $M$ has no isotropic points then its universal covering is isometric to a Riemannian product.
\end{thm}

We conclude with a brief account of our methods.  A $\cvc(\e)$ three-manifold $M$ with extremal curvature $\e$ has the following local structure in neighborhoods of non-isotropic points (see Section \ref{section:notation}):  $M$ admits a local orthonormal framing $\{ e_1, e_2, e_3 \}$ that diagonalizes the Ricci tensor with eigenvalues $\{ \l+\e, \l+\e, 2\e \}$ where $\l \neq \e$. The line field spanned by $e_3$ is globally defined in the set of non-isotropic points and is tangent to a foliation by complete geodesics.  This is essentially all that local theory yields. Indeed, a necessary condition for $M$ to be locally homogeneous is for $\l$ to be locally constant.  Even in this case, the local isometry classes of Riemannian three-manifolds for which the eigenvalues of the Ricci tensor are constants $\rho_1=\rho_2 \neq \rho_3$ depend on two arbitrary functions of one variable \cite{ko1, ko2}.

We introduce global methods by studying various ode along the $e_3$-geodesics satisfied by Christoffel symbols. When $M$ has finite volume and  $\e=-1$ these ode give enough information to conclude that, in a suitable frame, the curvature and its covariant derivatives are constant. A result of Singer \cite{si} then implies that $M$ is locally homogeneous. Similarly, when $M$ has finite volume and $\e=0$, these ode give enough information to deduce the local product structure. 

Section \ref{section:notation} introduces preliminary ideas common throughout the paper. Section \ref{section:homogeneous} describes the homogeneous $\cvc(\e)$ three-manifolds with extremal curvature $\e$.  The proofs of Theorems \ref{maintheorem-case-1} and \ref{maintheorem-rr} appear in Section \ref{section:cvc(-1)}.  The proof of Theorem \ref{maintheorem-case0} appears in Section \ref{section:cvc(0)}.

A preliminary study of $\cvc(1)$ three-manifolds with extremal curvature $1$ shows that there are large moduli of such metrics.  The construction of these moduli involve substantially different methods and will constitute a future paper.

\section{\bf General structure of $\cvc(\e)$ three-manifolds}
\label{section:notation}

In this section, we set up notation and collect together general results that will be used in subsequent sections.

\subsection{\bf Local computations}

We begin with a curvature identity that holds for an \textit{arbitrary} Riemannian three-manifold $M$.  Let $\{v_1, v_2, v_3 \}$ be an orthonormal framing of a neighborhood $U$ of a point $p \in M$. For unit orthogonal vectors $X, Y \in T_pM$, write $X = \sum_{i=1}^3 a_i v_i$ and $Y = \sum_{i=1}^3 b_i v_i$, where the $a_i$ and $b_i$ are scalars.  Set $A = (a_1, a_2, a_3)$ and $B = (b_1, b_2, b_3)$. Then $|A| = |B| = 1$. Since $X$ and $Y$ are orthogonal unit vectors, we have:
$$
\sec(X \wedge Y) = \langle R(X, Y) Y, X \rangle .
$$

\begin{lem}\label{local3d}
Let $C = (c_1, c_2, c_3) = A \times B$ and set $Z = \sum_{i=1}^3 c_i v_i$. Then
$$
\langle R(X, Y) Y, X \rangle = \frac{1}{2} S - \Ric(Z,Z),
$$
where $\Ric$ is the Ricci curvature and $S$ is the scalar curvature.
\end{lem}

\begin{proof}
We first compute:
\begin{eqnarray*}
R(X,Y) &=& \sum_{i,j} a_i b_j R(v_i,v_j) = \sum_{i,j} a_i b_j R_{ij} \\
&=& (a_1b_2 - a_2 b_1) R_{12} + (a_1b_3 - a_3 b_1) R_{13} + (a_2b_3 - a_3 b_2) R_{23}\\
&=& c_3 R_{12} - c_2 R_{13} + c_1 R_{23},
\end{eqnarray*}

where $R_{ij} : T_pM \to T^*_pM$. Then

\begin{eqnarray*}
\langle R(X, Y) Y, X \rangle
&=& c_3^2 R_{1221} - c_2 c_3 R_{1231} - c_1 c_3 R_{1223}\\
&&- c_3 c_2 R_{1321} + c_2^2 R_{1331} - c_1 c_2 R_{1332} \\
&&- c_1 c_3 R_{2312} - c_1 c_2 R_{2331} + c_1^2 R_{2332} 
\end{eqnarray*}

Using $c_1^2 + c_2^2 + c_3^2 = 1$ we have

\begin{eqnarray*}
\langle R(X, Y) Y, X \rangle
&=& (1 - c_1^2 - c_2^2) R_{1221} - c_2 c_3 R_{1231} - c_1 c_3 R_{1223}\\
&&- c_3 c_2 R_{1321} +  (1 - c_1^2 - c_3^2) R_{1331} - c_1 c_2 R_{1332} \\
&&- c_1 c_3 R_{2312} - c_1 c_2 R_{2331} + (1 - c_2^2 - c_3^2)  R_{2332}  \\
&=& \frac{1}{2} S -   \begin{pmatrix} c_1& c_2 & c_3 \\   \end{pmatrix}  \begin{pmatrix} \Ric  \end{pmatrix} \begin{pmatrix} c_1 \\
c_2 \\
c_3 \end{pmatrix}\\
&=&  \frac{1}{2} S - \Ric(Z,Z).
\end{eqnarray*}
\end{proof}

\begin{lem}\label{quadratic}
Let $\{v_1,v_2,v_3\}$ be a local orthonormal framing over a neighborhood $U$ of $p \in M$ for which the Ricci tensor is diagonal.  Let $\sigma$ be a two-plane at $p$ and $Z=\sum_{i=1}^{3} c_i v_i$ a unit vector orthogonal to $\sigma$.  Then $$\sec(\sigma)=c_1^2 \l_{23} + c_2^2 \l_{13} + c_3^2 \l_{12}$$ where $$
\l_{12} = R_{1221}, \;\; \l_{13}= R_{1331}, \;\; \l_{23} = R_{2332}.  
$$
\end{lem}

\begin{proof}

As $\{v_1,v_2,v_3\}$ diagonalizes the Ricci tensor,
\begin{eqnarray*}
\frac{1}{2} S &=& \l_{12} + \l_{13} + \l_{23} \\
\Ric(Z,Z) &=& c_1^2 ( \l_{12} + \l_{13}) + c_2^2( \l_{12} +  \l_{23}) + c_3^2 ( \l_{13} + \l_{23}).
\end{eqnarray*}

By Lemma \ref{local3d},  $$\sec(\sigma)=\frac{1}{2}S-\Ric(Z,Z)=c_1^2 \l_{23} + c_2^2 \l_{13} + c_3^2 \l_{12},$$ as desired.

\end{proof}

\begin{lem}
\label{cvclocal}
Let $M$ be a Riemannian three-manifold with $\cvc(\e)$ and let $\{v_1,v_2,v_3\}$ be a local orthonormal framing over a neighborhood of $p \in M$ which diagonalizes the Ricci tensor.  Then at least one of the sectional curvatures $\l_{ij}$ equals $\e$.
\end{lem}

\begin{proof}
If not, then there either exist two of the $\l_{ij}$ which are both greater than $\e$ or which are both less than $\e$.  Consider the case when two of the $\l_{ij}>\e$ (the other case is analogous).  After a possible relabeling of indices, we have that  both $\l_{23}$ and $\l_{13} $ are greater than $\e$.  Then by Lemma \ref{quadratic}, all planes containing $v_3$ have curvature greater than $\e$, a contradiction.
\end{proof}

\subsection{Adapted frames, the $A$ matrix, and curvature equations.}\text{}

The starting point for all of our analysis concerning $\cvc(\e)$ manifolds with extremal curvature $\e$ is the following:

\begin{thm}
\label{thm:local-character}
Let $M$ be a Riemannian three-manifold with $\cvc(\e)$ and extremal curvature $\e$. Then for each non-isotropic point $p \in M$ there is a number $\l_p\neq \e$ and a unit vector $e \in T_pM$ such that 

\begin{enumerate}

\item the curvature $\e$ planes at $p$ are precisely those containing the vector $e$

 \item the plane $P$ orthogonal to $e$ has sectional curvature $\l_p$ and all sectional curvatures at $p$ lie between $\e$ and $\l_p$

\item for any orthonormal frame $\{e_1, e_2 \}$ of the plane $P$, the frame $\{e_1, e_2, e= e_3 \}$ diagonalizes the Ricci tensor at $p$ with eigenvalues $\{ \l + \e, \l+ \e, 2\e \}$.

\end{enumerate}
\end{thm}

\begin{proof}
Let $p \in M$ be a non-isotropic point and assume that $\sec \geq \e$;  the case when $\sec \leq \e$ is handled similarly.  Let $\{v_1, v_2, v_3\}$ be an orthonormal framing of $T_pM$ as in Lemma \ref{quadratic}. It suffices to prove that exactly two of $\l_{12},\l_{13},\l_{23}$ equal $\epsilon$.  

As $p$ is not isotropic, one of the $\l_{ij}$ is not equal to $\e$ and by Lemma \ref{cvclocal} at least one of the $\l_{ij}$ equals $\e$.  If at most one of the $\l_{ij}$ equals $\e$, then one is less than $\e$ and one is greater than $\e$.  This contradicts the assumption that $\e$ is extremal. 
\end{proof}

By Theorem \ref{thm:local-character}, there is a continuous function $\lambda:M \rightarrow \R$, (with value $\e$ at isotropic points), with the property that $\sec$ has range between $\l_p$ and $\e$ at each $p\in M$.  In fact, we have that $\lambda$ is smooth on $M$ since the scalar curvature of $M$ equals $2\lambda+4 \e$.  Throughout, we let ${\mathcal I}=\{\l=\e\}$ denote the closed subset of isotropic points and ${\mathcal P} = M \setminus {\mathcal I}$ the open subset of non-isotropic points. 

We frequently will exploit the fact that $\cvc(\e)$ metrics lift to $\cvc(\e)$ metrics on covers.  In particular, we always assume that $M$ is oriented.  On the non-isotropic set ${\mathcal P}$ the vector field $e_3$ is \textit{locally} well-defined; the line field spanned by the vector field $e_3$ is globally defined.  This line field may not be orientable in which case it is in a double cover of each connected component of  ${\mathcal P}$.  In all arguments that follow, the reader may check that there is no loss of generality by passing to these covers when necessary. Consequently, we will always assume that the line field is oriented by a \textit{globally} defined vector field $e_3$ on $\mathcal{P}$.  As $M$ is oriented, a choice for the vector field $e_3$ as above orients the perpendicular plane field $e_3^{\perp}$.  

\begin{defn} An \textit{adapted frame} is a positively oriented orthonormal framing $\{e_1,e_2,e_3\}$.  
\end{defn}

Note that any two adapted frames differ by the $\SO(2)$ action on positively oriented orthonormal subframings of $e_3^{\perp}$.  Making a choice of such a subframe, we write the Christoffel symbols of the Levi-Civita connection $\n$ for the adapted frame $\{e_1, e_2, e_3\}$ as follows:

\begin{eqnarray}
\label{equ:Christoffel} \nonumber
\n_{e_1} e_3 = a_{11} e_1 + a_{12} e_2  && \n_{e_2} e_3 = a_{21} e_1 + a_{22} e_2 \\[.2cm] \nonumber
\n_{e_3} e_1 = c e_2 && \n_{e_3} e_2 = -c e_1 \\[.2cm]
\n_{e_2} e_1 = f e_2  - a_{21} e_3 && \n_{e_2} e_2 = -f e_1  - a_{22} e_3 \\[.2cm]  \nonumber
\n_{e_1} e_2 = ge_1  - a_{12} e_3 && \n_{e_1} e_1 = -g e_2  - a_{11} e_3 \\[.2cm] \nonumber
\n_{e_3} e_3 = 0. \nonumber
\end{eqnarray}

We remark that $\n_{e_3} e_3 = 0$ is demonstrated below in Theorem \ref{thm:e_3-geodesics}.  Consider the linear endomorphism $e_3^{\perp} \rightarrow e_3^{\perp}$ defined by $v \mapsto \nabla_{v} e_3$.  Throughout, we denote its adjoint by $A:e_{3}^\perp \rightarrow e_3^{\perp}$. With respect to the subframing $\{e_1,e_2\}$ of $e_3^{\perp}$ we have that 

$$
A= \begin{pmatrix} a_{11} & a_{12} \\
a_{21} & a_{22} \end{pmatrix} 
$$

\bigskip

Let $T_{\theta}=\begin{pmatrix} \cos \theta & \sin \theta \\
-\sin \theta & \cos \theta \end{pmatrix}$.  Under the orthonormal change of  adapted frame:

\begin{equation}
\label{equ:frametransform}
 \begin{pmatrix} \widetilde{e}_1 \\
 \widetilde{e}_2 \end{pmatrix}   \\  = T_{\theta}  \\ \begin{pmatrix}{e_1} \\
 {e_2} \end{pmatrix}   \\ 
\end{equation}

the matrix A transforms by:

\begin{equation}
\label{equ:transform-A}
\begin{pmatrix} \widetilde{a_{11}} & \widetilde{a_{12}} \\
\widetilde{a_{21}} & \widetilde{a_{22}} \end{pmatrix} = T_{\theta} \begin{pmatrix} a_{11} & a_{12} \\
a_{21} & a_{22} \end{pmatrix} T_{\theta}^{-1}
\end{equation}

and the remaining Christoffel symbols $f$, $g$, and $c$  transform by:

\begin{equation}
\label{equ:transform-f-g-c}
 \begin{pmatrix} \widetilde{f} \\
 \widetilde{g} \\
   \widetilde{c} \\ \end{pmatrix}   \\  = \begin{pmatrix} \cos \theta & \sin \theta & 0 \\
-\sin \theta & \cos \theta & 0 \\
0 & 0 & 1\\ \end{pmatrix}   \\ \begin{pmatrix} f + {e_2}(\theta) \\
 g - {e_1}(\theta) \\
 c  +  {e_3}(\theta) \\ \end{pmatrix}  \\ 
\end{equation}

\bigskip

Next, we describe a normal forms decomposition for $A$.  Write $A$ as the sum of a symmetric matrix $A_{sym}$ and a skew-symmetric matrix $A_{sk}$.  Further decompose  $A_{sym}$ into the sum of a traceless symmetric matrix $A_{0}$ and the multiple of the identity matrix $(\frac{1}{2} \tr A) \mbox{Id}$ so that 
$$
A = A_{0} + (\tfrac{1}{2} \tr A) \mbox{Id} + A_{sk}.
$$
\medskip

Set $\sigma = \frac{1}{2}(a_{11}- a_{22})$, $\tau = \frac{1}{2} (a_{12} + a_{21})$, $a = \frac{1}{2} (a_{11} + a_{22})$ and $b= \frac{1}{2} (a_{12} - a_{21})$. 

\medskip

Then,
$$
A_0=  \begin{pmatrix} \sigma & \tau \\
\tau & -\sigma \end{pmatrix}, \,\,\,\,\, (\tfrac{1}{2} \tr A) \mbox{Id}= \begin{pmatrix} a & 0 \\0 & a \end{pmatrix},\,\,\,\,\,  A_{sk}=\begin{pmatrix} 0 & b \\-b & 0 \end{pmatrix}. $$

The matrices $(\tfrac{1}{2} \tr A)\mbox{Id}$ and $A_{sk}$ are fixed under conjugation by an element of $\SO(2)$.  It follows that the normal forms decomposition of $A$ is invariant under conjugation by an element of $\SO(2)$ and that $\tr A$, $\det A$, $a$, and $b$ are smooth functions on $\mathcal{P}$ defined independently of a choice of an adapted frame.   Note that under the change $e_3 \to -e_3$ we have $b \to -b$ and $a \to -a$.

Under conjugation by the matrix $$T_{\theta}  =  \begin{pmatrix} \cos \theta &  \sin \theta \\
- \sin  \theta &  \cos  \theta \end{pmatrix}$$ the matrix $A_0$ transforms according to

\begin{equation}
\label{equ:transformA_0}
T_{\theta}  A_{0} T_{\theta}^{-1} =  \begin{pmatrix} \sigma \cos 2 \theta + \tau \sin 2  \theta &  \tau \cos 2 \theta - \sigma \sin 2 \theta \\[.2cm]
\tau \cos 2  \theta - \sigma \sin 2 \theta & -(\sigma \cos 2 \theta + \tau \sin 2  \theta) \end{pmatrix}   \\
\end{equation}

\medskip

While $\sigma$ and $\tau$ in general depend on a choice of an adapted frame, the function $-\det A_0= \sigma^2+ \tau^2$ is a smooth function on $\mathcal{P}$ defined independently of a choice of adapted frame.

Using the curvature tensor we derive the following equations on the Christoffel symbols:
From $R_{1221} = \l$:
\begin{equation}
\label{equ:curvature--1221}
 e_1(f) + e_2(g)  + f^2 + g^2 + c(a_{12} - a_{21}) + \det A =- \l
\end{equation}
From $R_{1331} = \e$:
\begin{equation}
\label{equ:curvature--1331}
-e_3(a_{11})  +  c(a_{12} + a_{21}) - (a_{11})^2 - a_{12} a_{21} = \e
\end{equation}
From $R_{2332} = \e$ :
\begin{equation}
\label{equ:curvature--2332}
 -e_3(a_{22})  - c( a_{12} + a_{21}) - (a_{22})^2 - a_{12} a_{21} = \e
\end{equation}
From $R_{1213} = 0$ :
\begin{equation}
\label{equ:curvature--1213}
-e_1(a_{21}) + e_2(a_{11}) + g(a_{11} - a_{22}) - f(a_{21} + a_{12})= 0
\end{equation}
From $R_{1223} = 0$ :
\begin{equation}
\label{equ:curvature--1223}
-e_1(a_{22}) + e_2(a_{12}) +  f(a_{11}-a_{22}) + g(a_{12} + a_{21})= 0
\end{equation}
From $R_{1312} = 0$ :
\begin{equation}
\label{equ:curvature--1312}
e_1(c) + e_3(g) + ga_{11} + f(c - a_{12})= 0
\end{equation}
From $R_{2312} = 0$ :
\begin{equation}
\label{equ:curvature--2312}
e_2(c) - e_3(f)   - f a_{22} + g(c + a_{21})= 0
\end{equation}
From $R_{1323} = 0$ :
\begin{equation}
\label{equ:curvature--1323}
 e_3(a_{12}) +  c( a_{11} - a_{22}) + a_{12} \tr A = 0
\end{equation}
From $R_{2313} = 0$ :
\begin{equation}
\label{equ:curvature--2313}
 e_3(a_{21}) +  c( a_{11} - a_{22}) + a_{21} \tr A = 0
\end{equation}

\medskip

Hence taking the difference of (\ref{equ:curvature--1323}) and (\ref{equ:curvature--2313}) we have:
\begin{equation}
\label{equ:curvature--difference}
e_3(b) +  (\tr A)b = 0
\end{equation}
Taking the sum of (\ref{equ:curvature--1323}) and (\ref{equ:curvature--2313}) we have:
\begin{equation}
\label{equ:curvature--sum2}
e_3(\tau) + 2c \sigma +  (\tr A) \tau = 0
\end{equation}
Taking the difference of (\ref{equ:curvature--1331}) and (\ref{equ:curvature--2332}) we have:
\begin{equation}
\label{equ:curvature--difference2}
e_3(\sigma) - 2c \tau +  (\tr A) \sigma = 0
\end{equation}

 \bigskip

 \subsection{The $e_3$-geodesic decomposition of $\mathcal{P}$.}\text{}

 The first theorem in this subsection demonstrates that $\mathcal{P}$ is foliated by geodesics of $M$ tangent to the vector field $e_3$.  Throughout, these geodesics will be called $e_3$-geodesics.

\begin{thm}
\label{thm:e_3-geodesics}
On the set $\mathcal{P}$ the vector-field $e_3$ satisfies: $\n_{e_3} e_3 = 0$.
\end{thm}  

\begin{proof}
We use the differential Bianchi identity twice. First
$$
\langle(\n_{e_2}R)(e_1, e_3)e_3, e_1 \rangle + \langle(\n_{e_1}R)(e_3, e_2)e_3, e_1 \rangle + \langle(\n_{e_3}R)(e_2, e_1)e_3, e_1 \rangle=0
$$ 
and the curvature conditions:  $$\langle R(e_1, e_3)e_3, e_1 \rangle = \e=\langle R(e_2, e_3)e_3, e_2 \rangle,$$ $$\langle R(e_1, e_2)e_2, e_1 \rangle = \lambda,$$and $$\langle R(e_i, e_j)e_k, e_\ell \rangle = 0,$$ when three of $i,j,k,\ell$ are distinct, imply that:
$$
(\l - \e) \langle \n_{e_3} e_3, e_2 \rangle = 0.
$$
Next, 
$$
\langle(\n_{e_1}R)(e_2, e_3)e_3, e_2 \rangle + \langle(\n_{e_2}R)(e_3, e_1)e_3, e_2 \rangle + \langle(\n_{e_3}R)(e_1, e_2)e_3, e_2 \rangle=0,
$$ 
 implies that:
$$
(\l - \e) \langle \n_{e_3} e_3, e_1 \rangle = 0.
$$
The result follows.
\end{proof}

\medskip

Next, we describe a decomposition of $\mathcal{P}$ into two sets.  Let $\mathcal{P}_1=\{-\det A_0>0\}$ and $\mathcal{P}_2=\{\det A_0=0\}$.  Clearly, $\mathcal{P}$ is the disjoint union of the open subset $\mathcal{P}_1$ and the closed subset $\mathcal{P}_{2}$.  Recall that a subset of a foliated space is said to be \textit{saturated} if it is a union of entire leaves of the foliation.

\begin{lem}
\label{saturation}
The subsets $\mathcal{P}_1$ and $\mathcal{P}_2$ are saturated.
\end{lem}

\begin{proof}
It suffices to prove that $\mathcal{P}_2$ is a saturated subset.  Let $p \in \mathcal{P}_2$ and let $\g(t)$ be the  $e_3$-geodesic with $\g(0)=p$ and $\dot{\g}(0)=e_3$. Let $\det A_0(t)=\det A_0(\g(t))$.  By (\ref{equ:curvature--sum2}) and (\ref{equ:curvature--difference2}),  $\det A_0(t)$ satisfies the ode $$\frac{d}{dt}(\det A_0)=-2\tr A(\det A_0).$$  As $\det A_0(0)=0$, it follows that $\det A_0(t) \equiv 0$, concluding the proof.

\end{proof}

Note that at points $p \in  \mathcal{P}_2$, the matrix $A_0=0$ independent of the choice of an adapted frame.  This immediately implies:

\begin{lem}
\label{type2frame}
At each point $p \in  \mathcal{P}_2$ and for any choice of adapted frame at $p$, $$A=\begin{pmatrix} a & b \\
-b & a \end{pmatrix}. $$
\end{lem}

In contrast, at points $p \in \mathcal{P}_1$ the entries of the matrix $A_0$ depend on a choice of an adapted frame.  In subsequent sections, the following lemma is used to pick out a particular adapted framing at points in $\mathcal{P}_1$.

\begin{lem}
\label{type1frame}
At each point $p \in \mathcal{P}_1$, there are precisely two adapted framings for which the matrix $A_0$ has the form $$A_0=\begin{pmatrix} \sigma & 0 \\
0 & -\sigma \end{pmatrix}$$ with $\sigma=\sqrt{-\det A_0}>0$.  Moreover, these two adapted framings differ by the element $T(\pi) \in \SO(2)$.
\end{lem}

\begin{proof}
Let $p \in \mathcal{P}_1$ and fix a positively oriented framing $\{e_1,e_2\}$ of $e_3^{\perp}$ at $p$.  Suppose that with respect to this framing $$A_0=\begin{pmatrix} \sigma & \tau \\
\tau & -\sigma \end{pmatrix}$$ where $\sigma^2+\tau^2>0$.  We first claim that there is a unique $\theta \in [0,\pi)$ such that $T_{\theta}A_{0}T_{\theta}^{-1}$ has the desired form.  By (\ref{equ:transformA_0}) this is equivalent to proving that 

\begin{equation}\label{equality}
\tau \cos 2\theta -\sigma \sin 2 \theta=0
\end{equation} 
and 
\begin{equation}\label{inequality}
\sigma \cos 2 \theta + \tau \sin 2\theta>0
\end{equation} have a unique common solution $\theta \in [0,\pi)$.

First suppose that $\tau=0$.  Since $\sigma^2+\tau^2 >0$, we have that $\sigma \neq 0$.  By (\ref{equality}) $\sin 2\theta=0$ so that $\theta=0$ or $\theta=\frac{\pi}{2}$.  By (\ref{inequality}), $\theta=0$ is the unique common solution when $\sigma>0$ and $\theta=\frac{\pi}{2}$ is the unique common solution when $\sigma<0$.

Next suppose that $\tau>0$.  Equation (\ref{equality}) yields $$\cot 2 \theta=\frac{\sigma}{\tau}$$ which has precisely two solutions $\theta \in [0,\pi)$, one with $\theta \in (0,\frac{\pi}{2})$ and one with $\theta \in (\frac{\pi}{2},\pi)$.  It is easy to check that the former satisfies (\ref{inequality}) while the latter does not.  An analogous line of reasoning shows that there is a unique common solution when $\tau<0$, concluding the proof of our claim.

By the previous claim, we may assume that the positively oriented framing $\{e_1,e_2\}$ of $e_3^{\perp}$ from the beginning of the proof puts $A_0$ in the desired form $$A_0=\begin{pmatrix} \sigma & 0 \\
0 & -\sigma \end{pmatrix}$$ with $\sigma=\sqrt{- \det A_0}>0$.  To conclude the proof, we show that for  $\theta \in (0,2\pi)$, the matrix $T_{\theta}A_0T_{\theta}^{-1}=A_0$ if and only if $\theta=\pi$.  By (\ref{equality}), we necessarily have that $\sin 2\theta=0$ so that $\theta=\frac{\pi}{2}$, $\pi$, or $\frac{3\pi}{2}$.  Inequality (\ref{inequality}) additionally holds if and only if $\cos 2\theta =1$ or equivalently if and only if $\theta=\pi$.
\end{proof}

 \subsection{Evolution equations for $\tr A$, $\det A$, and $\l$ along $e_3$-geodesics.}\text{}

Restrict the three scalar functions $\tr A$, $\det A$ and $\lambda = R_{1221}$ on $\mathcal{P}$ to functions along an $e_3$-geodesic $\g$. Let $t$ be a parameter for $\g$ such that $e_3 = \frac{d}{dt}$. Then,

\begin{thm}
\label{thm:basic-ode}
Along the $e_3$-geodesic $\g(t)$ we have:
\begin{eqnarray}
\label{equ:ode-lambda}
&&\frac{d}{dt}  (\l-\e) =  - \tr A  (\l-\e)\\ \label{equ:ode-trace}
&&\frac{d}{dt}(\tr A) =2(\det A -\e) - (\tr A)^2  \\ \label{equ:ode-det}
&&\frac{d}{dt}  (\det A + \e) = - \tr A  (\det A + \e) 
\end{eqnarray}
\end{thm}

\begin{proof}
To prove the first equality, use the differential Bianchi identity:
$$
\langle(\n_{e_3}R)(e_1, e_2)e_2, e_1 \rangle + \langle(\n_{e_1}R)(e_2, e_3)e_2, e_1 \rangle + \langle(\n_{e_2}R)(e_3, e_1)e_2, e_1 \rangle=0
$$ 
and the curvatures $\langle R(e_1, e_2)e_2, e_1 \rangle = \lambda$, $\langle R(e_1, e_3)e_3, e_1 \rangle = \e$, $\langle R(e_2, e_3)e_3, e_2 \rangle = \e$, and $\langle R(e_i, e_j)e_k, e_\ell \rangle = 0$, if any three of $i,j,k,\ell$ differ to conclude that:
$$
e_3(\l) = - \tr A (\l -\e).
$$
The first equality follows.

\medskip

To prove the second equality, add (\ref{equ:curvature--1331}) and (\ref{equ:curvature--2332}) to derive

$$
e_3(a_{11} + a_{22}) =  - (a_{11})^2  - (a_{22})^2 - 2a_{12} a_{21} -2\e,
$$ from which the second equality easily follows.

\medskip

To prove the third equality, compute $e_3(a_{11}a_{22}-a_{12}a_{21})$ using (\ref{equ:curvature--1323}) and (\ref{equ:curvature--2313}) to derive $e_3(\det A) = - \tr A (\e + \det A)$. The result follows.

\end{proof}

The following are two immediate corollaries:

\begin{cor}\label{foliation}
Along any $e_3$-geodesic, $\l$ never assumes the value $\e$. Hence, the vector field $e_3$ is complete and ${\mathcal P}$ is foliated by complete $e_3$-geodesics. 
\end{cor}

\begin{cor}\label{tr-det}
Along any $e_3$-geodesic, $\tr A \equiv 0$ if and only if $\det A\equiv \e.$ 
\end{cor}
\medskip

We conclude this section with a general form solution to the coupled odes (\ref{equ:ode-trace}) and (\ref{equ:ode-det}).  Explicit form solutions when $\e=-1$ or $\e=0$ are given at the beginning of each relevant section.

\begin{thm}
\label{odesolution}

Along the $e_3$-geodesic  $\g$ let $\ell(t)$ denote the solution to
$$
\ell''+4\e \ell=2k, \,\,\, k=\det A(0)+\e
$$ 
with initial conditions 
$$
\ell(0)=1,\,\,\,\ell'(0)=\tr A(0).
$$

Then $\ell(t)>0$ for each $t \in \R$ and
$$\tr A=\frac{\ell'}{\ell}
$$

and 

$$
\det A + \e =\frac{1}{\ell}( \det A(0) +\e).
$$

\end{thm}

\begin{proof}
To solve (\ref{equ:ode-trace}) and (\ref{equ:ode-det}) introduce the function $\ell(t) = \exp(\int_0^t (\tr A) d\t)$ so that $\tr A = \frac{\ell'}{\ell}$, $\ell(0)=1$, $\ell'(0)=\tr A(0)$, and $\ell(t)>0$ for all $t\in \R$.  Then from (\ref{equ:ode-det}) we have:
\begin{equation}
(\det A + \e)' = -\frac{\ell'}{\ell} (\det A + \e)
\end{equation}
Hence,
\begin{equation}
\ell (\det A + \e)= k,
\end{equation}
where $k = \det A(0) + \e$ is a constant. 
From (\ref{equ:ode-trace}) we have:
\begin{equation}
\frac{\ell'' \ell - (\ell')^2}{\ell^2}= \bigg( \frac{\ell'}{\ell}\bigg)' = 2(\frac{k}{\ell}-2\e) -\frac{(\ell')^2}{\ell^2}\end{equation}

Hence,
\begin{equation} 
\frac{\ell''}{\ell} = \frac{2k}{\ell}-4\e
\end{equation}
Simplifying, we have:
\begin{equation}
\label{equ:solving2}
(\ell)''+ 4 \e \ell = 2k
\end{equation}

\end{proof}

By (\ref{equ:curvature--difference}), (\ref{equ:ode-lambda}), and (\ref{equ:ode-det}),  $b$, $\l-\e$, and $\det A+\e$ satisfy the same linear ode along an $e_3$-geodesic $\g$. Hence:

\begin{cor}\label{lambda-det-b}
For each $e_3$-geodesic $\g$, there exists constants $K$ and $C$ such that $\det A+\e=K(\l-\e)$ and $b=C(\l-\e)$.
\end{cor}

\section{\bf Homogeneous $\cvc(\e)$ three-manifolds with extremal curvature $\e$}
\label{section:homogeneous}

A good general reference for this section is \cite{mi}.

In this section, $M$ denotes a connected, simply-connected, complete, and homogeneous three-manifold with $\cvc(\e)$ and extremal curvature $\e$.  We continue to use the notation of Section \ref{section:notation}.  

Sekigawa \cite{se} classified the connected, simply-connected, complete, and homogeneous three-manifolds.  They either have constant sectional curvatures, or are isometric to a product $\Sigma \times \R$ where $\Sigma$ is a complete and simply-connected surface with constant sectional curvatures, or are isometric to a connected and simply-connected three-dimensional Lie group endowed with a left-invariant metric.

If $M$ has an isotropic point, then all points are isotropic by homogeneity and $M$ is a space form of curvature $\e$.  The products $S^2 \times \R$ and $\H^2 \times \R$ have no isotropic points and $\cvc(0)$.  It remains to classify the left-invariant metrics on connected and simply-connected three dimensional Lie groups with $\cvc(\e)$, extremal curvature $\e$, and no isotropic points. 

Let $G$ denote a connected and simply-connected three-dimensional Lie group endowed with a left-invariant Riemannian metric determined by an inner-product $\langle \cdot,\cdot \rangle$ on its Lie algebra $\mathfrak{g}$. We assume that $G$ has $\cvc(\e)$, extremal curvature $\e$, and no isotropic points.  By Theorem \ref{thm:local-character} there exists an orthonormal framing $\{\bar{e}_1, \bar{e}_2, e_3\}$ of $\mathfrak{g}$ that diagonalizes the Ricci tensor where the vector $e_3$ is contained in all curvature $\epsilon$ planes.  As left-translation acts by orientation preserving isometries of $G$, $\{\bar{e}_1,\bar{e}_2,e_3\}$ extends to a left-invariant adapted framing on all of $G$.

As this framing is left-invariant, the Christoffel symbols (\ref{equ:Christoffel}) are \textit{constant} functions on $G$.  As $\lambda-\epsilon$ is a non-zero constant function on $G$,  Theorem \ref{thm:basic-ode} implies that $\tr A=0$ and $\det A =\epsilon$.  Next, we consider the normal forms for $A$ and in particular, $a$, $b$, $\sigma$ and $\tau$ as defined in Section \ref{section:notation}.

We have that  $a=\frac{1}{2} \tr A =\frac{1}{2}(a_{11}+a_{22})=0$. Consequently, $a_{11}=-a_{22}$ and $\sigma=\frac{1}{2}(a_{11}-a_{22})=a_{11}=-a_{22}$.  By possibly replacing $e_3$ with $-e_3$, we will assume that $b=\frac{1}{2}(a_{12}-a_{21}) \geq 0$.    By Lemmas \ref{type2frame} and \ref{type1frame}, we may assume that $\tau=0$, or equivalently that $a_{12}=-a_{21}$.  Consequently, $b=a_{12}=-a_{21}\geq 0$ and moreover, with respect to the left-invariant adapted framing $\{\bar{e}_1,\bar{e}_2,e_3\}$ we have that $$A=  \begin{pmatrix} \sigma & b  \\
 -b & -\sigma  \end{pmatrix}  \\
$$  where $b\geq 0$, $\sigma>0$ if $G=\mathcal{P}_1$, and $\sigma=0$ if $G=\mathcal{P}_2$.
\medskip

The curvature equations (\ref{equ:curvature--1221})-(\ref{equ:curvature--2313}) for this left-invariant adapted framing reduce to:

\bigskip

From $R_{1221} = \l$:
\begin{equation}
\label{equ:curvature--1221-hom}
  f^2 + g^2 + 2cb+ \epsilon =- \l
\end{equation}
From $R_{1331} = \epsilon$:
\begin{equation}
\label{equ:curvature--1331-hom}
 b^2-\sigma^2=\epsilon
\end{equation}
From $R_{2332} = \epsilon$ :
\begin{equation}
\label{equ:curvature--2332-hom}
 b^2-\sigma^2 = \epsilon
\end{equation}
From $R_{1213} = 0$ :
\begin{equation}
\label{equ:curvature--1213-hom}
2g \sigma= 0
\end{equation}
From $R_{1223} = 0$ :
\begin{equation}
\label{equ:curvature--1223-hom}
2f \sigma= 0
\end{equation}
From $R_{1312} = 0$ :
\begin{equation}
\label{equ:curvature--1312-hom}
g\sigma + f(c - b)= 0
\end{equation}
From $R_{2312} = 0$ :
\begin{equation}
\label{equ:curvature--2312-hom}
f\sigma + g(c -b)= 0
\end{equation}
From $R_{1323} = 0$ :
\begin{equation}
\label{equ:curvature--1323-hom}
 2c\sigma = 0
\end{equation}
From $R_{2313} = 0$ :
\begin{equation}
\label{equ:curvature--2313-hom}
 2c\sigma = 0
\end{equation}

\medskip

Consider the linear transformation $L:\mathfrak{g} \rightarrow \mathfrak{g}$ defined by extending $$L(\bar{e}_1)=[\bar{e}_2,e_3],\,\,\,\, L(\bar{e}_2)=[e_3,\bar{e}_1],\,\,\,\, L(e_3)=[\bar{e}_1,\bar{e}_2]$$ linearly.  Then the Lie group $G$ is unimodular if and only if $L$ is self-adjoint \cite{mi}.  In this case, the Lie algebra $\mathfrak{g}$ is determined by the signs of the eigenvalues of $L$ as described by \cite[p. 307]{mi}.   

With respect to the framing $\{\bar{e}_1,\bar{e}_2,e_3\}$ described above, we have that 

$$
L=  \begin{pmatrix} c+a_{21} & -\sigma & g \\
 -\sigma & c-a_{12} & -f \\
0 & 0 & -2b \end{pmatrix}  \\
$$

Hence $G$ is unimodular if and only if $f=g=0$.  Our first lemma characterizes when $G$ Riemannian covers a finite volume manifold:

\begin{lem}
\label{finvolquot}
Assume that $G$ is a connected and simply-connected three-dimensional Lie group endowed with a left-invariant metric of $\cvc(\e)$, extremal curvature $\e$, and with no isotropic points .  Let $\{\bar{e}_1, \bar{e}_2, e_3\}$ be a left-invariant framing of $G$ as described above.  Then $G$ Riemannian covers a finite volume manifold if and only if $f=g=0$. 
\end{lem}

\begin{proof}
First assume that $f=g=0$ or equivalently that $G$ is unimodular.  As all three-dimensional unimodular Lie groups admit lattices \cite{mi}, $G$ covers a finite volume manifold. 

 Next assume that $G$ covers a finite volume manifold $N$.  Then $\Gamma= \pi_1(N)$ acts by isometries on $G$.  After possibly passing to an index two subgroup, we may assume that $\Gamma$ acts by orientation preserving isometries of $G$ with finite volume quotient.  We assume that one of $f$ or $g$ is non-zero and derive a contradiction in what follows.
 
 Let $\I^{+}(G)$ denote the group of orientation preserving isometries of $G$ and $F<\I^{+}(G)$ the subgroup consisting of orientation preserving isometries fixing the identity element $e\in G$.  We first claim that $F$ is trivial or isomorphic to $\Z_2$.  To see this, note that since $G$ has constant vector curvature and is not a space form, the derivative map of each $I \in F$ fixes the line in $\mathfrak{g}\cong T_e G$ spanned by $e_3$.  Define the homomorphism $\phi:F \rightarrow \{-1,1\}$ by $dI(e_3)=\phi(I)(e_3)$.  The claim will follow once we prove $\phi$ is injective. 
 
 To see that $\phi$ is injective, assume that $I$ is an orientation preserving isometry of $G$ fixing the identity $e$ and the vector $e_3\in \mathfrak{g}\cong T_e G$.  As $I$ preserves orientation, there is a $\theta \in [0,2\pi)$ such that $$dI(\bar{e}_1)=\cos\theta\,\bar{e}_1+\sin \theta \,\bar{e}_2$$ $$dI(\bar{e}_2)=-\sin \theta \, \bar{e}_1 +\cos \theta\, \bar{e}_2.$$  Using the fact that $dI([\bar{e}_1,\bar{e}_2])=[dI(\bar{e}_1),dI(\bar{e}_2)]$, we have that 
 
$$  \begin{pmatrix} g \\
f \end{pmatrix}  \\  =  \begin{pmatrix} \cos \theta & \sin \theta \\
-\sin \theta & \cos \theta \end{pmatrix}   \\   \begin{pmatrix} g \\
f \end{pmatrix}  \\$$

As one of $g$ or $f$ is assumed to be non-zero, we have that $dI$ is the identity map of $\mathfrak{g}$.  An isometry is determined by where it sends a point and its derivative at this point.  Consequently, $I$ is the identity map of $G$, concluding the proof that $\phi$ is injective.  

Left-translation $L:G \rightarrow \I^{+}(G)$, $g \mapsto L_g$, identifies $G$ with a subgroup of  $\I^{+}(G)$.  Define the map $$\nu:\I^{+}(G) \rightarrow \{-1,1\}$$ by $\nu(I)=\phi(L^{-1}_{I(e)}\circ I)$ for each $I\in \I^{+}(G)$.  Alternatively, the derivative of each isometry $I \in \I^{+}(G)$ maps the left-invariant vector-field $e_3$ to the left-invariant vector-field $\nu(I) e_3$.  It follows easily that $\nu$ is a homomorphism and that $G=\ker (\nu)$.  The group $\hat{\Gamma}=G\cap \Gamma$ has index at most two in $\Gamma$, hence also acts on $G$ with a finite volume quotient $\hat{N}$.  As all elements in $\hat{\Gamma}$ act by left-translations, the left-invariant vector fields $\bar{e}_1$ and $\bar{e}_2$ on $G$ descend to vector fields $\hat{e}_1$ and $\hat{e}_2$ on $\hat N$.  A simple calculation shows that $f\equiv\di{\bar{e}_1}$ and $g \equiv \di{\bar{e}_2}$ and consequently, $f\equiv \di{\hat{e}_1}$ and $g\equiv \di {\hat{e}_2}$.  As $\hat{N}$ has finite volume, the divergence of every vector field on $\hat{N}$ vanishes somewhere, contradicting the assumption that one of $f$ or $g$ is non-zero.
 \end{proof}

We begin with the proof of the following:

\begin{thm}\label{cvc1-hom}
Assume that $M$ is a connected,  simply-connected, complete and homogeneous $\cvc(1)$ three-dimensional manifold with extremal curvature $+1$.
Then either

\begin{enumerate}

\item $M$ is isometric to a left-invariant metric on $\SU(2)$, or

\item $M$ is isometric to a left-invariant metric on the Heisenberg group, or

\item $M$ is isometric to a left-invariant metric on $\widetilde{\SL(2,\R)}$, or

\item $M$ is isometric to a left-invariant metric on a non-unimodular solvable three-dimensional Lie group.  
\end{enumerate}

The manifolds appearing in $(1)-(3)$ Riemannian cover finite-volume manifolds while those in $(4)$ do not.  
\end{thm}

\begin{proof}
As the round three dimensional sphere is isometric to a left-invariant metric on $\SU(2)$, we may assume that $M$ is isometric to a connected and simply-connected three dimensional Lie group $G$ endowed with a left-invariant metric with no isotropic points (or equivalently $1\neq \lambda \in \R$).  We assume that $G$ is endowed with a left-invariant framing $\{\bar{e}_1,\bar{e}_2,e_3\}$ as described above. By (\ref{equ:curvature--1331-hom}) $b^2=1+\sigma^2$ so that $b \geq 1$ by our convention that $b\geq 0$.

In what follows, we consider two separate types of left-invariant metrics. \\\\

\textbf{Type I, $G=\mathcal{P}_1$:}  

Metrics of this type have $\sigma>0$.  Therefore $c=0$ by (\ref{equ:curvature--1323-hom}) and $$
 \begin{pmatrix} 0 \\
0 \end{pmatrix}  \\  =  \begin{pmatrix} \sigma & b \\
-b & -\sigma \end{pmatrix}   \\   \begin{pmatrix} g \\
-f \end{pmatrix}  \\
$$  by (\ref{equ:curvature--1312-hom}) and (\ref{equ:curvature--2312-hom}).  As $b^2-\sigma^2 =1$, $f=g=0$.  The curvature equations (\ref{equ:curvature--1221-hom})-(\ref{equ:curvature--2313-hom}) are all solved once values for the constants $b>1$ and $\sigma>0$ satisfying the equality $1=b^2-\sigma^2$ are specified.

Let $\alpha=\frac{\sigma}{b}$. As $1= b^2-\sigma^2=b^2(1-\alpha^2)$ and $b>1$, we have that  $\alpha$ may take any value in  $(0,1)$.  Let $\mu=b(1-\alpha)>0$.  A simple calculation shows that $\mu=(\frac{1-\alpha}{1+\alpha})^{1/2}$.  Note that $\mu$ may take any value in $(0,1)$ .  

Rotating $\bar{e}_1$ and $\bar{e}_2$ at all points by angle $\pi/4$, we obtain a new left-invariant framing $\{e_1, e_2, e_3\}$ with constant Christoffel symbols $c=f=g=a_{11}=a_{22}=0$ and $\mu=a_{12}=-\frac{1}{a_{21}}$.  The curvature $\l=-1$ by (\ref{equ:curvature--1221-hom}).  With respect to the framing $\{e_1,e_2,e_3\},$

$$
L=  \begin{pmatrix} -\frac{1}{\mu} & 0 & 0 \\
 0 & -\mu & 0 \\
0 & 0 & -(\mu+\frac{1}{\mu}) \end{pmatrix}.   \\
$$

According to \cite[p. 307]{mi}, $G$ is isomorphic to $\SU(2)$.

\bigskip

\textbf{Type II, $G=\mathcal{P}_2$ :} 

Metrics of this type have $\sigma=0$.  Thus $b^2=b^2-\sigma^2=1$ so that $b=1$ by our convention $b \geq 0$.  We consider two separate cases in what follows:
\medskip

\textbf{Case 1, $c\neq 1$:}

  By equations (\ref{equ:curvature--1312-hom}) and (\ref{equ:curvature--2312-hom}) we have that $f(c-1)=g(c-1)=0$ whence $f=g=0$ and $\lambda=-(2c+1)$.

With respect to the framing $\{\bar{e}_1,\bar{e}_2,e_3\}$, 

$$
L=  \begin{pmatrix} c-1 & 0 & 0 \\
 0 & c-1 & 0 \\
0 & 0 & -2 \end{pmatrix}   \\
$$ giving the bracket relations mentioned in the introduction.

According to \cite[p. 307]{mi}, if $c>1$ then $G$ is isomorphic to $\widetilde{\SL(2, \R)}$ and if $c<1$ then $G$ is isomorphic to $\SU(2)$.  

The interested reader may check that when $c=\frac{3}{2}$, $G$ is isometric to the universal covering of the unit-tangent bundle of the hyperbolic plane and that when $c=-1$, $G$ is isometric to the constant curvature one three sphere.  Moreover, the family of metrics with $c<1$ are isometric to the Berger spheres suitably rescaled so that the Hopf vector field is contained in curvature one planes.  The family of metrics with $c>1$ are constructed in a similar fashion starting with the universal covering of the unit-tangent bundle of the hyperbolic plane. \\\\
\medskip
\textbf{Case 2, $c=1$:}

As in the previous case, we have $b=1$.  The curvature equations $(\ref{equ:curvature--1221-hom})-(\ref{equ:curvature--2313-hom})$ are then satisfied for any values of $f$ and $g$.  The curvature $\lambda=-(3+f^2+g^2).$  With respect to the framing $\{\bar{e}_1,\bar{e}_2,e_3\}$, 

$$
L=  \begin{pmatrix} 0 & 0 & g \\
 0 & 0 & -f \\
0 & 0 & -2 \end{pmatrix}   \\
$$

The group $G$ is unimodular if and only if $f=g=0$ in which case $G$ is isomorphic to the three-dimensional Heisenberg group.  If one of $f$ or $g$ is non-zero, then $G$ is a non-unimodular three-dimensional Lie group, hence solvable \cite{mi}.  In this case, Lemma \ref{finvolquot} implies that $G$ does not cover a finite volume manifold. 
\end{proof}

\begin{rem}
\label{finer-classification1}
In Theorem \ref{cvc1-hom} the parameter $\mu \in (0,1)$ parameterizes the \textit{isometry classes} of Type I metrics and the parameter $c \in \R$ parameterizes the \textit{isometry classes} of Type II metrics as we now outline:
\medskip

For a Type II metric with parameter $c \in \R$ the tangent plane $e_3^{\perp}$ has curvature $-(2c+1)$ implying the second statement.  As for the first statement, let $\mu \in (0,1)$ and consider the metric Lie algebra $(\mathfrak{g}, <,>_{\mu})$ with basis $\{e_1,e_2,e_3\}$ satisfying:

$$ [e_2,e_3]=-\frac{1}{\mu} e_1, \,\,\,\,\,\,[e_3,e_1]=-\mu e_2, \,\,\,\,\,\,[e_1,e_2]=-(\frac{1}{\mu}+\mu)e_3.$$  Let $(G,<,>_{\mu})$ denote the induced connected and simply-connected Lie group with left-invariant metric.  As explained in the proof of Theorem \ref{cvc1-hom}, $G$ is isomorphic to $\SU(2)$.

Now suppose that $0<\mu_1<1$ and $0<\mu_2<1$.  For $i=1,2$, let 
$\{e^i_1,e^i_2,e^i_3\}$ denote the left-invariant orthonormal framing of $(SU(2), <,>_{\mu_i})$ just described.  Assume that $$I:(\SU(2),<,>_{\mu_1}) \rightarrow (\SU(2),<,>_{\mu_2})$$ is an isometry.  By following $I$ with a left-translation, we may assume that $I$ fixes the identity element.  The derivative map at the identity $dI: (\mathfrak{g},<,>_{\mu_1}) \rightarrow (\mathfrak{g},<,>_{\mu_2})$ is a linear isometry of metric Lie algebras.  It follows that 
 
 \begin{enumerate}
 \item $dI([e^1_1,e^1_2])=[dI(e^1_1),dI(e^1_2)],$  
 \item $dI(e_3^1)=\pm e_3^2,$ 
 \item  the image under $dI$ of the subframing $\{e_1^1,e_2^1\}$ of $(e_3^1)^{\perp}$ differs from the subframing $\{e_1^2,e_2^2\}$ of $(e_3^2)^{\perp}$ by an element of $O(2)$.
\end{enumerate}

A simple calculation using (1)-(3) implies $\mu_1=\mu_2$.  

It is interesting to note that as $\mu \rightarrow 1$, the metric Lie algebras $(\mathfrak{g},<,>_{\mu})$ converge to the Type II metric Lie algebra corresponding to the parameter $c=0$.    
\end{rem}

\begin{thm}\label{cvc-1-hom}
Assume that $M$ is a connected, simply-connected, complete and homogeneous $\cvc(-1)$ three-dimensional manifold with extremal curvature $-1$.
Then either 

\begin{enumerate}

\item $M$ is isometric to three dimensional hyperbolic space, or 

\item $M$ is isometric to a left-invariant metric on $E(1,1)$, or

\item $M$ is isometric to a left-invariant metric on $\widetilde{\SL(2,\R)}$.  

\end{enumerate}

The manifolds appearing in $(1)-(3)$ all Riemannian cover finite volume manifolds and the sectional curvatures of manifolds appearing in $(2)-(3)$ have range $[-1,1]$.  

\end{thm}

\begin{proof}
If $M$ is not hyperbolic, then $M$ is isometric to a connected and simply-connected three dimensional Lie group $G$ endowed with a left-invariant metric with no isotropic points (or equivalently $-1 \neq \lambda \in \R$).  We assume that $G$ is endowed with a left-invariant framing $\{\bar{e}_1,\bar{e}_2,e_3\}$ as described above. By (\ref{equ:curvature--1331-hom}),  $b^2-\sigma^2=-1$. In particular, $\sigma \neq 0$ and $G=\mathcal{P}_1$. 

By (\ref{equ:curvature--1323-hom}), $2c\sigma=0$ so that $c=0$. By (\ref{equ:curvature--1312-hom}) and (\ref{equ:curvature--2312-hom}), we have:

$$
 \begin{pmatrix} 0 \\
0 \end{pmatrix}  \\  =  \begin{pmatrix} \sigma & -b \\
-b & \sigma \end{pmatrix}   \\   \begin{pmatrix} f \\
g \end{pmatrix}  \\
$$ 

As $\sigma^2-b^2=1$, $f=g=0$.  By \ref{equ:curvature--1221-hom}, $\lambda=1$. 

Rotate $\bar{e}_1$ and $\bar{e}_2$ at all points by angle $\pi/4$.  By (\ref{equ:transform-A}) and (\ref{equ:transform-f-g-c}) we obtain a new left-invariant framing $\{e_1, e_2, e_3\}$ with constant Christoffel symbols $c=f=g=a_{11}=a_{22}=0$ and $a_{12}=\frac{1}{a_{21}}=b-\sigma\neq 0$.  Let $\mu=b-\sigma$.  Replacing $e_3$ with $-e_3$ if necessary, we may assume that $\mu>0$ and  by possibly switching $e_1$ and $e_2$ if necessary, we may assume that $\mu\geq 1$. With respect to the orthonormal framing $\{e_1,e_2,e_3\},$

$$
L=  \begin{pmatrix} \frac{1}{\mu} & 0 & 0 \\
 0 & -\mu & 0 \\
0 & 0 & \frac{1}{\mu}-\mu \end{pmatrix}   \\
$$

According to \cite[p. 307]{mi}, if $\mu=1$, then $G$ is isomorphic to $E(1,1)$ and if $\mu>1$, then $G$ is isomorphic to the universal covering group of $\SL(2,\R)$.  

As for the last claim of the theorem, it is well known that the three dimensional hyperbolic space covers finite volume manifolds while Lemma \ref{finvolquot} implies that manifolds appearing in (2) and (3) do as well.
\end{proof}

\medskip

\begin{rem}
\label{finer-classification-1}
An argument analogous to the one outlined in remark \ref{finer-classification1} classifies the \textit{isometry classes} of metrics appearing in Theorem \ref{cvc-1-hom}.  The interested reader may check that there is a unique isometry class of metrics appearing in (2) corresponding to the parameter $\mu=1$ and that the isometry classes of metrics appearing in (3) are parameterized by $\mu \in (1, \infty)$.  
\end{rem}

\begin{thm}\label{useinrigidity}
Assume that $M$ is a connected, simply-connected, complete and homogeneous three-manifold with extremal curvature $-1$.  If $M$ has positive hyperbolic rank, then $M$ is isometric to three dimensional hyperbolic space.
\end{thm}

\begin{proof}
The assumption that $M$ has positive hyperbolic rank implies that $M$ has $\cvc(-1)$.  From the proof of Theorem \ref{cvc-1-hom}, a connected, simply-connected, complete and homogeneous $\cvc(-1)$ three-manifold with extremal curvature $-1$ is isometric to $\mathbb{H}^3$ or to a three-dimensional Lie group $G$ admitting a left-invariant orthonormal framing $\{e_1,e_2,e_3\}$ satisfying

\begin{eqnarray}
\label{rankchristoffel}
 \nonumber
\n_{e_1} e_3 =  \mu e_2  && \n_{e_2} e_3 = \frac{1}{\mu} e_1 \\[.2cm] \nonumber
\n_{e_3} e_1 = 0 && \n_{e_3} e_2 = 0 \\[.2cm]
\n_{e_2} e_1 =   -\frac{1}{\mu} e_3 && \n_{e_2} e_2 = 0 \\[.2cm] \nonumber 
\n_{e_1} e_2 =   - \mu e_3 && \n_{e_1} e_1 =  0 \\[.2cm] \nonumber
\n_{e_3} e_3 = 0\nonumber
\end{eqnarray} where $\mu\geq1$ is a constant and $e_3$ lies in the intersection of all curvature $-1$ planes.  It suffices to demonstrate that a left-invariant metric with a left-invariant orthonormal framing as above does not have positive hyperbolic rank.  

Consider the unit-speed geodesic $\gamma:\R \rightarrow G$ defined by $\gamma(0)=e$ and $\dot{\gamma}(0)=e_1\in \mathfrak{g}$.  As $\nabla_{e_1} e_1=0$, $\dot{\gamma}(t)=e_1(\gamma(t))$ for all $t \in \R$.  We show that the geodesic $\gamma$ has no orthogonal Jacobi field $J(t)$ satisfying $\sec(\dot{\gamma}(t),J(t))=-1$ for all $t \in \R$ as follows.  If there were such a Jacobi field, then by (1) of Theorem \ref{thm:local-character}, $J(t)=a(t)e_3(t)$ for some smooth $a:\R \rightarrow \R$ with isolated zeroes.  Use (\ref{rankchristoffel}) to calculate $$J''(t)=2\mu a'(t)e_2(t)+(a''(t)-\mu^2a(t))e_3(t)$$ and $$R(J,\dot{\gamma})\dot{\gamma}(t)=-a(t)e_3(t).$$\\  The Jacobi equation $J''(t)+R(J,\dot{\gamma})\dot{\gamma}(t)=0$ implies that  $2\mu a'(t)=0$ and $a''(t)=(\mu^2+1)a(t).$  As $\mu\neq 0$, the equality $2\mu a'(t)=0$ implies that $a(t)$ is constant.  The equality $a''(t)=(\mu^2+1)a(t)$ then implies that $a(t)=0$ for all $t\in \R$, a contradiction.
\end{proof}

\begin{thm}
\label{cvc0-hom}
Assume that $M$ is a connected, simply-connected, complete and homogeneous $\cvc(0)$ three-dimensional manifold with extremal curvature $0$.
Then either 

\begin{enumerate}

\item $M$ is isometric to three-dimensional Euclidean space, or

\item $M$ is isometric to a product of space forms $S^2 \times \R$ or $\H^2 \times \R$

\item $M$ is isometric to a left-invariant metric on a non-unimodular solvable three-dimensional Lie group.  
\end{enumerate}

The manifolds appearing in $(3)$ do not Riemannian cover finite-volume manifolds.

\end{thm}

\begin{proof}
If $M$ is not isometric to a manifold in (1) or (2), then $M$ is isometric to a connected and simply-connected three-dimensional Lie group $G$ endowed with a left-invariant metric with no isotropic points (or equivalently $0\neq \lambda \in \R$).

We assume that $G$ is endowed with a left-invariant framing $\{\bar{e}_1,\bar{e}_2,e_3\}$ as described above. By (\ref{equ:curvature--1331-hom}), we have that $b^2=\sigma^2$.  We first claim that $b=\sigma=0$ or equivalently that $G=\mathcal{P}_2$.  Indeed, if this were not the case then (\ref{equ:curvature--1213-hom}), (\ref{equ:curvature--1223-hom}), and (\ref{equ:curvature--1323-hom}) imply that $f=g=c=0$.  By (\ref{equ:curvature--1221-hom}), $\lambda=0$, a contradiction.  Hence $\sigma=0$ and $b=0$.  The curvature equations (\ref{equ:curvature--1221-hom})-(\ref{equ:curvature--2313-hom}) are then satisfied for any values of $f$ and $g$ with $\lambda=-(f^2+g^2) \neq 0$.  With respect to the framing $\{\bar{e}_1,\bar{e}_2,e_3\}$, 

$$
L=  \begin{pmatrix} 0 & 0 & g \\
 0 & 0 & -f \\
0 & 0 & 0 \end{pmatrix}   \\
$$  Therefore, $G$ is a non-unimodular three dimensional solvable Lie group which does not cover a finite volume manifold by Lemma \ref{finvolquot}.  
\end{proof}

\medskip

As mentioned in the introduction, the eight Thurston geometries have constant vector curvature.  The hyperbolic, Euclidean, and spherical geometries obviously do.  The product geometries have constant vector curvature zero.  The $Nil$ and $\widetilde{\SL(2,\R)}$ geometries are Type II $\cvc(1)$ manifolds with $c=1$ and $c=\frac{3}{2}$, respectively. The $Sol$ geometry is the $\cvc(-1)$ manifold corresponding to $\mu=1$.

We conclude this section with a criterion due to Singer \cite{si} for a Riemannian manifold to be homogeneous that is used in Section \ref{section:cvc(-1)} to prove Theorem \ref{maintheorem-case-1}.  For an integer $n\geq 0$, a Riemannian manifold $M$ satisfies condition $P(n)$ if for each $x,y \in M$, there exists a linear isometry of $T_xM$ onto $T_yM$ which maps $(\nabla^{k}R)_x$ onto $(\nabla^{k}R)_y$ for each $k=0,1,\ldots, n$ where $R$ is the Riemannian curvature tensor and $\nabla$ is the Levi-Civita connection on $M$

\begin{thm}[Singer]
\label{singer}
Assume that $M$ is a connected, simply-connected, and complete Riemannian manifold satisfying condition $P(n)$ for sufficiently large $n$.  Then $M$ is Riemannian homogeneous.
\end{thm}

\section{\bf Three manifolds with $\cvc(-1)$ and extremal curvature $-1$}\label{section:cvc(-1)}

In this section, $M$ denotes a connected and complete three-manifold with $\cvc(-1)$, extremal curvature $-1$, and finite volume. We use the notation and conventions introduced in Section \ref{section:notation}. 

\medskip

Restrict the three scalar functions $\tr A$, $\det A$ and $\lambda = R_{1221}$ on $\mathcal{P}$ to functions along an $e_3$-geodesic $\gamma$. Let $t$ be a parameter such that $e_3 = \frac{d}{dt}$.  The reader may check the following:

\begin{lem}\label{odesolution:-1}
The solution to 
$$
\ell''-4 \ell=2k, \,\,\, k=\det A(0)-1
$$ 
with initial conditions 
$$
\ell(0)=1,\,\,\,\ell'(0)=\tr A(0)
$$

is

$$
\ell(t) = \frac{1}{2} \big( \tr A(0)  \sinh 2t  + (\det A(0) + 1) \cosh 2 t -  (\det A(0) - 1) \big).
$$
\end{lem}

\begin{thm}\label{tr=0:-1}
The functions $\tr A \equiv 0$ and $\det A \equiv -1$ on $\mathcal{P}$.  In particular, the flow generated by $e_3$ preserves volume.
\end{thm}

\begin{proof}
Note that $\di(e_3)=\tr A$.  Hence, the second statement follows from the first.  As for the first statement, it suffices to prove $\tr A$ vanishes identically along each $e_3$-geodesic by Corollary \ref{tr-det} .  For a given $e_3$-geodesic $\g(t)$, let $c_1=\tr A(0)$, $c_2=\det A(0)+1$, and $c_3=-2k=-2(\det A(0)-1)$ so that $$\ell(t)=\frac{1}{4}((c_1 + c_2) e^{2t} - (c_1 - c_2) e^{-2t} + c_3).$$  As $\ell$ is everywhere positive on $\g$, the initial conditions satisfy $c_1+c_2\geq0$ and $c_1-c_2\leq0$.  Theorem \ref{odesolution} implies

$$
\tr A(t)=\frac{\ell'(t)}{\ell(t)} = 2\;\; \frac{ (c_1 + c_2) e^{2t} + (c_1 - c_2) e^{-2t} } { (c_1 + c_2) e^{2t} - (c_1 - c_2) e^{-2t} + c_3}.
$$

First suppose that $(c_1+c_2)=0$ and $(c_1-c_2)=0$.  Then $c_3=4$ and the above formula implies $\tr A \equiv 0$.  Otherwise, if

\begin{eqnarray}
(c_1 + c_2) > 0 \;\; &\mbox{then} \;\; \tr A \to 2 & \; \mbox{as} \; t \to \infty, \\
(c_1 - c_2) < 0 \;\; &\mbox{then} \;\;  \tr A \to -2 &\; \mbox{as} \; t \to -\infty,\\
(c_1 + c_2) = 0,  (c_1 - c_2) < 0\;\; &\mbox{then} \;\;  \tr A \to 0 & \; \mbox{as} \; t \to \infty,\\
 (c_1 - c_2) = 0, (c_1 + c_2) > 0, \;\; &\mbox{then} \; \; \tr A \to 0 &\; \mbox{as} \; t \to- \infty.
\end{eqnarray}
Accordingly, $e_3$-geodesics in $\mathcal{P}$ fall into four disjoint classes. Those on which: (i) $\tr A \to 2$ as $t \to \infty$ and $\tr A \to -2$ as $t \to -\infty$, (ii)  $\tr A \to 2$ as $t \to \infty$ and $\tr A \to 0$ as $t \to -\infty$, (iii) $\tr A \to 0$ as $t \to \infty$ and $\tr A \to -2$ as $t \to -\infty$ and (iv) $\tr A \equiv 0$. We denote the set of points in ${\mathcal P}$ that lie on $e_3$-geodesics of type (i) by ${\mathcal S}_{-2,2}$, of type (ii) by ${\mathcal S}_{0,2}$, of type (iii) by ${\mathcal S}_{-2,0}$, and of type (iv) by ${\mathcal S}_{0,0}$. Our goal is to prove that $\mathcal{P}=\mathcal{S}_{0,0}$.  Note that by continuity of $\tr A$ on $\mathcal{P}$, it suffices to prove that each of ${\mathcal S}_{-2,2}$, ${\mathcal S}_{0,2}$ and ${\mathcal S}_{-2,0}$ has empty interior.

Seeking a contradiction, first suppose that ${\mathcal S}_{0,2}$ has non-empty interior.  As $\mathcal{S}_{0,2}$ is saturated by complete $e_3$-geodesics, we may find an open set $U\subset \mathcal{S}_{0,2}$ also saturated by complete $e_3$-geodesics and in particular invariant under the flow $\phi_t$ generated by the vector field $e_3$. By definition, an $e_3$-geodesic $\g \subset {\mathcal S}_{0,2}$ has a parameterization with initial conditions satisfying $(c_1 + c_2) > 0$ and $(c_1 - c_2) = 0$.  Consequently, the derivative $\ell'$ and $\di(e_3)=\tr A$ are positive functions on $\g$ and hence also on $\mathcal{S}_{0,2}$.

As $\vol(M) < \infty$, the set $U$ is a finite volume open set so that 
 $$\frac{d}{ds} \vol(\phi_s(U))=\int_{\phi_s(U)} \di(e_{3}).$$  This is a contradiction since the left-hand side equals $0$ by flow-invariance of $U$ and the right-hand side is positive since $\di(e_3)>0$ on $U=\phi_s(U)$, a non-empty open set.  An analogous argument proves that $\mathcal{S}_{-2,0}$ has no interior points.  It remains to prove that $\mathcal{S}_{-2,2}$ has empty interior.

Seeking a contradiction, suppose that $\mathcal{S}_{-2,2}$ has non-empty interior. Then we may find a closed two-dimensional disc $D\subset \mathcal{S}_{-2,2}$ transversal to the $e_3$-geodesic foliation.  For $x \in D$, let $\gamma_x$ denote the $e_3$-geodesic with $\gamma_x(0)=x$ and let $c_1(x)$ and $c_2(x)$ denote the corresponding initial conditions for $\ell$ along $\g_x$.  For $t \geq 0$, let $$U(t)=\{ \gamma_x(s) \, \vert \, (x,s) \in D \times [t, \infty)\, \}.$$  

It is easy to check that $\tr A (\g_x(t))>0$ for all $t$ satisfying $$e^{4t}>\frac{c_2(x)-c_1(x)}{c_1(x)+c_2(x)}.$$  As the initial conditions, $c_1(x)$ and $c_2(x)$ depend continuously on $x \in D$, there is a $T>0$ such that $\tr A$ is positive on all of $U(T)$.  

For $s \geq 0$, let $v_s=\vol(U(T+s))$.  The sets $U(T+s)$ have finite and positive volume.  When $0 \leq s_1 <s_2$, $U(T+s_2)\subset U(T+s_1)$ so that $v_s$ is a finite non-increasing function.  On the other hand,
$$\frac{d}{ds}v_s=\int_{U(T+s)} \di(e_{3})$$ is positive since $\di(e_3)=\tr A>0$ on $U(T+s)$, a set with non-empty interior.  This contradiction concludes the proof.
\end{proof}

\bigskip

\begin{cor}
$\mathcal{P}=\mathcal{P}_1$
\end{cor}

\begin{proof}
Suppose not and choose a point $p \in \mathcal{P}_2$.  By Lemma \ref{type2frame} we have that 
$$A = \begin{pmatrix} a & b \\
-b & a \end{pmatrix}$$ with respect to any adapted framing at the point $p$.   By Theorem \ref{tr=0:-1} we have that $a=\frac{1}{2}\tr A=0$ and $\det A =b^2=-1$, a contradiction.  
\end{proof}

By Lemma \ref{type1frame}, at each point $p \in \mathcal{P}=\mathcal{P}_1$ there are precisely two adapted framings $\{e_1,e_2,e_3\}$ and $\{-e_1,-e_2,e_3\}$ with respect to which 
\begin{equation}\label{-1Aform1}
A = \begin{pmatrix} \sigma & b \\
-b & -\sigma \end{pmatrix}
\end{equation} with $\sigma>0$.  

Let $L_1$ denote the line field on $\mathcal{P}$ spanned by $e_1$.  For a connected component $\mathcal{C}$ of $\mathcal{P}$,  we let $\bar{\mathcal{C}}$ denote a connected component of the orientation double cover of the line field $L_1$ endowed with the lifted $\cvc(-1)$ metric.  A choice of orientation for $L_1$ on $\bar{\mathcal{C}}$ yields a global adapted framing $\{e_1,e_2,e_3\}$ of $\bar{\mathcal{C}}$ for which the matrix $A$ has the form (\ref{-1Aform1}).

\begin{lem}
For a global adapted framing $\{e_1,e_2,e_3\}$ of $\bar{\mathcal{C}}$ with respect to which $A$ has form (\ref{-1Aform1}), $c\equiv 0$.  Moreover, $b$ and $\sigma$ are constant along $e_3$-geodesics.
\end{lem}

\begin{proof}
As $\tau\equiv 0$ and $\sigma>0$ on $\bar{\mathcal{C}}$, (\ref{equ:curvature--sum2}) implies $c \equiv 0$.  By Theorem \ref{tr=0:-1}, $\tr A=0$ so that by (\ref{equ:curvature--difference}) and (\ref{equ:curvature--difference2}), $b$ and $\sigma$ are constant along $e_3$-geodesics.
\end{proof}

By (\ref{equ:transform-A}) and (\ref{equ:transform-f-g-c}), rotating the framing $\{e_1,e_2\}$ of $e_3^{\perp}$ at all points in $\bar{\mathcal{C}}$ by angle $\frac{\pi}{4}$ yields a global adapted framing $\{\bar{e}_1,\bar{e}_2,e_3\}$ of $\bar{\mathcal{C}}$ with respect to which 
\begin{equation}\label{-1Aform2}
A = \begin{pmatrix} 0 & \mu \\
\frac{1}{\mu} & 0 \end{pmatrix}
\end{equation} where $\mu=\sigma-b\neq 0$ is constant along $e_3$-geodesics and $c\equiv 0$.

\begin{thm}\label{frame:-1}
For a global adapted framing $\{\bar{e}_1,\bar{e}_2,e_3\}$ of $\bar{\mathcal{C}}$ with respect to which $A$ has form (\ref{-1Aform2}), $f\equiv g \equiv  0$.
\end{thm}

\begin{proof}
Let $\g(t)$ be an $e_3$-geodesic.  Using $a_{11} = a_{22} = c = 0$ in (\ref{equ:curvature--1312}) and (\ref{equ:curvature--2312}) we have:
$$
 \begin{pmatrix} f' \\
g' \end{pmatrix}  \\  =  \begin{pmatrix} 0 & \frac{1}{\mu} \\
\mu & 0 \end{pmatrix}   \\   \begin{pmatrix} f \\
g \end{pmatrix}  \\
$$ along $\g$.
Hence 
$$
 \begin{pmatrix} f \\
g \end{pmatrix}  \\  =  c_1 e^t \begin{pmatrix} 1 \\
\mu  \end{pmatrix} + c_2 e^{-t}  \\   \begin{pmatrix} 1 \\
-\mu \end{pmatrix}  \\.
$$ where $c_1=\frac{1}{2}(f(0)+f'(0))$ and $c_2=\frac{1}{2}(f(0)-f'(0))$.  Since $\mu$ is a non-zero constant along $\g$, $f \equiv 0$ along $\g$ if and only $g \equiv 0$ along $\g$.

Seeking a contradiction, we assume that $f$ is non-zero at a point $p \in \bar{\mathcal{C}}$.  Let $D$ denote a closed two-dimensional disc transverse to the vector field $e_3$ and passing through $p$.  By continuity of $f$, we may assume that $f$ is non-zero and bounded on $D$.  For each $x \in D$, let $\g_x$ denote the $e_3$-geodesic with $\g_x(0)=x$.  

The initial conditions $c_1(x)=\frac{1}{2}(f(x)+f'(x))$ and $c_2(x)=\frac{1}{2}(f(x)-f'(x))$ for $f$ along $\g_x$ depend continuously on $x \in D$.  As $f(p)\neq 0$, one of $c_1(p)$ or $c_2(p)$ is non-zero.  If $c_1(p)\neq 0$, then after possibly shrinking $D$, we have that $c_1(x)\neq 0$ for all $x \in D$.  It follows that for each $x \in D$, $|f| \rightarrow \infty$ exponentially along $\g_x$ as $t \rightarrow \infty$.  This contradicts the fact that $f$ is bounded on $D$ since by Theorem \ref{tr=0:-1} and Poincare recurrence, $\g_x$ returns to $D$ along a sequence of times tending to infinity for almost every $x \in D.$ 

If $c_2(p)\neq 0$, an analogous argument yields a contradiction, completing the proof.

\end{proof}

\begin{prop}\label{mu:constant}
The function $\mu$ is a  constant $\neq 0$ and $\l \equiv 1$ on $\bar{\mathcal{C}}$ .
\end{prop}

\begin{proof}
Let $\{\bar{e}_1,\bar{e}_2,e_3\}$ be a global adapted framing of $\bar{\mathcal{C}}$ with respect to which $A$ has form (\ref{-1Aform2}).  From (\ref{equ:curvature--1213}) and (\ref{equ:curvature--1223}) it follows that $\bar{e}_1(\frac{1}{\mu})=0$ and $\bar{e}_2(\mu) = 0$. Since $e_3(\mu)=0$ it follows that
$\mu$ is a global constant. By (\ref{equ:curvature--1221}) we have that $\l\equiv 1$ on $\bar{\mathcal{C}}$.
\end{proof}

\begin{cor}\label{empty-all}
Either ${\mathcal P} = \emptyset$ or ${\mathcal P}= M$.
\end{cor}

\begin{proof}
Note that Proposition \ref{mu:constant} implies $\l=1$ on all of $\mathcal{P}$.  This implies that $\mathcal{P}$ is closed since $\l=-1$ on the isotropic set $\mathcal{I}=M \setminus \mathcal{P}$.  The result follows since $\mathcal{P}$ is also open in $M$, a connected manifold.  
\end{proof}

We conclude this section with the proofs of Theorems \ref{maintheorem-case-1} and \ref{maintheorem-rr} from the introduction.

\begin{proof}[Proof of Theorem~\ref{maintheorem-case-1}]
By Theorem \ref{cvc-1-hom}, it suffices to prove that if $M$ has finite-volume, $\cvc(-1)$, extremal curvature $-1$ and is not real-hyperbolic, then the universal covering of $M$ is homogeneous.  

If $M$ is not real-hyperbolic, then $\mathcal{P} \neq \emptyset$.  By Corollary \ref{empty-all}, $M=\mathcal{P}$. Recall that on a connected component of a double cover $\bar{M}$ of $M$ there is an adapted framing $\{\bar{e}_1,\bar{e}_2,e_3\}$ of $\bar{M}$ for which the matrix $A$ has form (\ref{-1Aform2}).  With respect to this framing, the Christoffel symbols $a_{11}=a_{22}=c=0$.  By Theorem \ref{frame:-1}, the Christoffel symbols $f=g=0$ with respect to this framing.  By Proposition \ref{mu:constant} the remaining Christoffel symbols $a_{12}$ and $a_{21}$ are constant with respect to this framing.  An application of Theorem \ref{singer} proves that the universal covering of $M$ is homogeneous.
\end{proof}

\begin{proof}[Proof of Theorem \ref{maintheorem-rr}]
As $M$ has positive hyperbolic rank, $M$ has $\cvc(-1)$.  The conclusion holds by Theorems \ref{maintheorem-case-1} and \ref{useinrigidity}.
\end{proof}

\section{\bf Three manifolds with $\cvc(0)$ and extremal curvature $0$}\label{section:cvc(0)}

In this section, $M$ denotes a connected and complete three-manifold with $\cvc(0)$, extremal curvature $0$, and finite volume. We use the notation and conventions introduced in Section \ref{section:notation}. 

\medskip

Restrict the three scalar functions $\tr A$, $\det A$ and $\lambda = R_{1221}$ on $\mathcal{P}$ to functions along an $e_3$-geodesic $\gamma$. Let $t$ be a parameter such that $e_3 = \frac{d}{dt}$.  The reader may check the following:

\begin{lem}\label{odesolution:0}
The solution to 
$$
\ell''=2k, \,\,\, k=\det A(0)
$$ 
with initial conditions 
$$
\ell(0)=1,\,\,\,\ell'(0)=\tr A(0)
$$

is

$$
\ell(t) =\det A(0) t^2 + \tr A(0) t+ 1
$$
\end{lem}

\begin{thm}\label{tr=0:0}
The functions $\tr A \equiv 0$ and $\det A \equiv 0$ on $\mathcal{P}$.  In particular, the flow generated by $e_3$ preserves volume.
\end{thm}

\begin{proof}
Note that $\di(e_3)=\tr A$.  Hence, the second statement follows from the first.   As for the first statement, it suffices to prove $\tr A$ vanishes identically along each $e_3$-geodesic  $\g$ by Corollary \ref{tr-det} .  Theorem \ref{odesolution} implies

$$
\tr A(t)=\frac{\ell'(t)}{\ell(t)} =  \frac{2 \det A(0) t +  \tr A(0) } {  \det A(0)  t^2 + \tr A(0) t  + 1}.
$$

Corollaries \ref{foliation} and \ref{lambda-det-b} imply that if $\det A(0) = 0$ then $\det A(t) \equiv 0$.  In this case, the fact that $\ell$ is positive on $\g$ implies that $\tr A(0)=0$, and consequently that $\tr A(t) \equiv 0$.  Otherwise, we have that $\det A(0) \neq 0$ and $\det A$ never vanishes on $\g$.

Accordingly we define two classes of $e_3$-geodesics. Those on which: (i) $\det A$ never vanishes, and those on which (ii) $\det A$ and $\tr A$ identically vanish. These classes are disjoint. We denote the set of points in ${\mathcal P}$ that lie on $e_3$-geodesics of type (i) by ${\mathcal S}$ and of type (ii) by ${\mathcal T}$.  We will conclude the proof by showing $\mathcal{P}=\mathcal{T}$ or equivalently that $\mathcal{S}=\emptyset.$

Suppose $\mathcal{S} \neq \emptyset$. Then $\mathcal{S}$ is a non-empty open set saturated by $e_3$-geodesics on which $\det A$ never vanishes.  Let $D$ denote a closed disc transversal to the foliation of ${\mathcal S}$ by $e_3$-geodesics.  For $x \in D$, let $\g_x$ denote the $e_3$-geodesic  with $\g_x(0)=x$.  For $(x,t)\in D \times \R$, let $$\det A(x,t) = \det A (\g_x(t))\,\,\,\,\, \tr A(x,t)=\tr A (\g_x(t)).$$  For $t>0$, let $$U(t)=\{\gamma_x(s) \, \vert \, (x,s) \in D \times [t,\infty)\, \}.$$

We first claim that $\det A$ is a positive function on $\mathcal{S}$.  To see this, let $\g$ be an $e_3$-geodesic in $\mathcal{S}$.  As $\ell>0$ on $\g$, we have that $\det A(0)>0$.  The claim follows since $\det A$ is continuous and does not vanish on $\g$. It follows that for $x \in D$,  $\tr A(x,t) >0$ for all $t > -\frac{\tr A(x,0)} {2\det A(x,0)}$.  As the initial conditions $\det A(x,0)$ and $\tr A(x,0)$ vary continuously with $x \in D$, there exists a $T>0$ such that $\tr A$ is a positive function on $U(T)$.

For $s\geq0$, let $v_s=\vol(U(T+s)).$  As $U(T+s)$ contains an open set, $v_s>0$ and as $\vol(M) < \infty$, $v_s < \infty$.  When $0\leq s_1<s_2$, $U(T+s_2) \subset U(T+s_1)$ whence $v_s$ is a finite non-increasing function.  On the other hand, $$\frac{d}{ds} v_{s} = \int_{U(T+s)} \di(e_3)$$ and since $\di(e_3)=\tr A>0$ on $U(T+s)$, a set with interior, $v_s$ is strictly increasing.  This contradiction completes the proof.

\end{proof}

\bigskip

We first consider the subset $\mathcal{P}_1\subset \mathcal{P}$.  

By Lemma \ref{type1frame}, at each point $p \in \mathcal{P}_1$ there are precisely two adapted framings $\{e_1,e_2,e_3\}$ and $\{-e_1,-e_2,e_3\}$ with respect to which 
\begin{equation}\label{0Aform1}
A = \begin{pmatrix} \sigma & b \\
-b & -\sigma \end{pmatrix}
\end{equation} with $\sigma>0$.

Let $L_1$ denote the line field on $\mathcal{P}_1$ spanned by $e_1$.  For a connected component $\mathcal{C}$ of $\mathcal{P}_1$,  we let $\bar{\mathcal{C}}$ denote a connected component of the orientation double cover of the line field $L_1$ endowed with the lifted $\cvc(0)$ metric.  A choice of orientation for $L_1$ on $\bar{\mathcal{C}}$ yields a global adapted framing $\{e_1,e_2,e_3\}$ of $\bar{\mathcal{C}}$ for which the matrix $A$ has the form (\ref{0Aform1}).

\begin{lem}
For a global adapted framing $\{e_1,e_2,e_3\}$ of $\bar{\mathcal{C}}$ with respect to which $A$ has form (\ref{0Aform1}), $c\equiv0$.  Moreover, $b$ and $\sigma$ are constants along $e_3$-geodesics.
\end{lem}

\begin{proof}
As $\tau=0$ and $\sigma >0$ on $\bar{\mathcal{C}}$, (\ref{equ:curvature--sum2}) implies $c \equiv 0$.  By Theorem \ref{tr=0:0}, $\tr A=0$ so that by (\ref{equ:curvature--difference}) and (\ref{equ:curvature--difference2}), $b$ and $\sigma$ are constant along $e_3$-geodesics.
\end{proof}

 By Theorem \ref{tr=0:0}, $\det A =0$ so that $\sigma^2=b^2$ on $\bar{\mathcal{C}}$.  As $\sigma$ does not vanish on the connected set $\bar{\mathcal{C}}$, it follows that either $\sigma=b$ on all of $\bar{\mathcal{C}}$ or $\sigma=-b$ on all of $\bar{\mathcal{C}}$.  By (\ref{equ:transform-f-g-c}) and (\ref{equ:transformA_0}), rotating the framing $\{e_1,e_2\}$ of $e_3^{\perp}$ at all points in $\bar{\mathcal{C}}$ by angle $\pm \frac{\pi}{4}$ (depending on whether $\sigma=\pm b$, yields a new global adapted framing $\{\bar{e}_1,\bar{e}_2,e_3\}$ of $\bar{\mathcal{C}}$ with respect to which 
\begin{equation}\label{0Aform2}
A = \begin{pmatrix} 0 & 0 \\
\mu & 0 \end{pmatrix}
\end{equation}  where $\mu=-(b+\sigma)$ if $\sigma=b$, $\mu=\sigma-b$ if $\sigma=-b$, and $c \equiv 0$ on $\bar{\mathcal{C}}$.  In particular, $\mu$ is non-zero and constant along $e_3$-geodesics in $\bar{\mathcal{C}}$.

\begin{lem}\label{rho-constant}
For an adapted framing $\{\bar{e}_1, \bar{e}_2, e_3\}$ of $\bar{\mathcal{C}}$ with respect to which the matrix $A$ has form \ref{0Aform2}, the function $\mu$ is a  constant $\neq 0$ on $\bar{\mathcal{C}}$.
\end{lem}

\begin{proof}
 With respect to the given adapted framing $\{\bar{e}_1, \bar{e}_2, e_3\}$, we have that $a_{11}=a_{22}=a_{12}=c=0$. Moreover, $a_{21}=\mu$ is a non-zero smooth function on $\bar{\mathcal{C}}$ satisfying $e_3(\mu)=0$.  

From (\ref{equ:curvature--1223}) it follows that $g=0$ on $\bar{\mathcal{C}}$.  Consider the function $k=\frac{f}{\mu}$ on $\bar{\mathcal{C}}$.  From (\ref{equ:curvature--2312}) and the fact that $g=0$ it follows that $e_3(f)=0$.  Since additionally $e_3(\mu)=0$, the function $k$ is constant along $e_3$-geodesics.


By way of contradiction we suppose that $\mu$ is not constant on $\bar{\mathcal{C}}$.  We may then find an interval $I\subset \R$ consisting of regular values of $\mu$.  For a regular value $r \in I$, we let $\Sigma_r = \mu^{-1}(r)$ denote the smooth level surface.  Since $e_3(\mu) =0$  the vector field $e_3$ is everywhere tangent to $\Sigma_r$. Therefore, each $\Sigma_r$ is foliated by complete $e_3$-geodesics.   

Consider the subset $$X=\{x \in \bar{C}\, \vert\, k(x)=0\}=\{x \in \bar{C}\, \vert\, f(x)=0\}$$ of $\bar{C}$.  We claim that $X$ has no interior point in $\bar{C}$.  Indeed, otherwise there exists an open subset $O$ of $\bar{C}$ with $O \subset X$.  As $f=0$ on $O$,  (\ref{equ:curvature--1221}) implies that $\l=0$ on $O$, contradicting the fact that $\bar{C}$ consists of non-isotropic points.

As $X$ has no interior points, we may find a closed two dimensional disc $D \subset \mu^{-1}(I)$ transverse to the vector field $e_3$ on which $k$ does not vanish.  For each $x \in D$, let $\g_x$ denote the $e_3$-geodesic with $\g_x(0)=x$ and let $$U=\{\g_x(t)\,\vert\, (x,t) \in D\times \R\}.$$  As $k$ is constant on $e_3$-geodesics, $k$ does not vanish on $U$.  By (\ref{equ:curvature--1213}), we have that $$\bar{e}_1(\mu)=-f \mu=-k\mu^2.$$  Therefore, the vector field $\bar{e}_1$ is not tangent to $\Sigma_r\cap U$ for each $r \in \mu(U)$. 

We construct a new adapted framing $\{ \tilde{e}_1, \tilde{e}_2, e_3 \}$ on $U$ as follows. The vector field $ \tilde{e}_1$ is the unit-normal vector field to the level sets $\Sig_r$ lying on the same side of these level sets as $\bar{e}_1$.  The orientation of $e_3^{\perp}$ then determines the vector field $\tilde{e}_2$ which is necessarily tangent to the level sets $\Sig_r$.  Using that each $\Sig_r$ is a surface,

\begin{equation}
\label{equ:bracket1-zero}
\langle [\tilde{e}_2, e_3] , \tilde{e}_1 \rangle = 0.
\end{equation}

For each $p \in U$ there is a unique angle $\theta(p) \in (-\frac{\pi}{2},\frac{\pi}{2})$, depending smoothly on $p \in U$,  such that $T_{\theta(p)}\in SO(2)$ rotates the subframing $\{\bar{e}_1, \bar{e}_2\}$ of  $e_3^{\perp}$ to the subframing $\{ \tilde{e}_1,  \tilde{e}_2 \}$.  The matrix $\tilde{A}$ is given by (\ref{equ:transform-A}):
$$
\tilde{A}= T A T^{-1} =  \begin{pmatrix} \mu \cos \theta \sin \theta &  -  \mu \sin^2 \theta \\
 \mu \cos^2 \theta & - \mu \cos \theta \sin \theta \end{pmatrix}   \\
$$
In particular $\tilde{a}_{21} =  \mu \cos^2 \theta$ and by (\ref{equ:transform-f-g-c}) $\tilde{c} = e_3(\theta)$.
Then using (\ref{equ:bracket1-zero})
\begin{equation}
 \mu \cos^2 \theta =\tilde{a}_{21}=  \langle \n_{\tilde{e}_2} {e_3}, \tilde{e}_1 \rangle = \langle \n_{{e}_3} {\tilde{e}_2},  \tilde{e}_1 \rangle =-\tilde{c}= -e_3(\theta).
\end{equation}

Let $\g$ be an $e_3$-geodesic  lying in $U$.  We have along $\g$:
\begin{equation}
\label{equ:relating-rho-theta}
-\mu  = \frac{e_3(\theta)}{ \cos^2 \theta} = \frac{\theta'}{ \cos^2 \theta},
\end{equation}
Integrating along $\g$ from $0$ to $t$ and recalling that $\mu$ is a non-zero constant along $\g$ we have:
$$
-\mu t = \int_0^t \frac{\theta'}{ \cos^2 \theta} = \tan \theta(t)-\tan \theta(0).
$$
Hence,
$$
\theta(t) = \tan^{-1}( \tan \theta(0)-\mu t).
$$  

Theorem \ref{tr=0:0} and Poincare recurrence imply that for almost every $x\in D$, there exists a sequence of times $t_i \rightarrow \infty$ such that $\g_x(t_i) \in D.$  By the last equation, $\theta(\g_x(t_i)) \rightarrow \pm \frac{\pi}{2}$.  However, by compactness of $D$, $\theta$ is bounded away from $\pm \frac{\pi}{2}$ on $D$, a contradiction.

\end{proof}

\begin{lem}\label{cvc(0):f=g=0}
For a global adapted framing $\{\bar{e}_1,\bar{e}_2,e_3\}$ of $\bar{\mathcal{C}}$ with respect to which $A$ has form (\ref{0Aform2}), $f \equiv g \equiv 0.$
\end{lem}

\begin{proof}
In this framing, $a_{11}=a_{22}=a_{12}=0$ and $\mu=a_{21}$ is a non-zero constant.  Therefore, (\ref{equ:curvature--1213}) and (\ref{equ:curvature--1223}) reduce to $\mu f=0$ and $\mu g= 0$.  The result follows. 
\end{proof}

\begin{cor}\label{A=0}
${\mathcal P}_1 = \emptyset$ and hence ${\mathcal P} = {\mathcal P}_{2}$
\end{cor}

\begin{proof}
If $\mathcal{P}_1$ is non-empty, then we may choose a component $\mathcal{C}\subset \mathcal{P}_1$ as above.  In an adapted framing of $\bar{\mathcal{C}}$ which puts $A$ in the form (\ref{0Aform2}) we have $f=g=c=0$ so that by (\ref{equ:curvature--1221}) $\l =\det A = 0$  on $\bar{\mathcal{C}}$.  Therefore, $\l=0$ on $\mathcal{C}$, a contradiction. 
\end{proof}

We conclude with the proof Theorem \ref{maintheorem-case0} from the introduction.

\begin{proof}[Proof of Theorem \ref{maintheorem-case0}]
Let $\mathcal{P}$ denote the set of non-isotropic points.  By standard proofs of de Rham's decomposition theorem, it suffices to prove that the line field on $\mathcal{P}$  spanned by $e_3$ is holonomy invariant.

By Corollary \ref{A=0}, $\mathcal{P}=\mathcal{P}_2$.  By Lemma \ref{type2frame} and Theorem \ref{tr=0:0}, the matrix $A$ is the zero matrix.  Consequently,  $\nabla_{(\cdot)} e_3$ vanishes on $e_3^{\perp}$.  As $e_3$ is also geodesic, $\nabla_{(\cdot)} e_3$ vanishes on all of $T \mathcal{P}$ as required.
\end{proof}

In Theorem \ref{maintheorem-case0}, the local product structure in the subset of non-isotropic points need not arise from a global product structure on the universal covering $\widetilde{M}$.  This is illustrated by examples of non-positively curved graph three-manifolds which have $\cvc(0)$ and irreducible universal covering.  Such examples play a similar role in \cite{guzh} where non-positively curved $n$-manifolds with bounded sectional curvatures and Ricci rank $r<n$ are shown to have a local product decomposition on the open subset of Ricci rank $r$ points.  In contrast, we do not know examples of $\cvc(0)$ three-manifolds of non-negative curvature with an irreducible universal covering.  By the Cheeger-Gromoll splitting theorem, this reduces to the question of whether the three-sphere admits a $\cvc(0)$ metric of non-negative curvature, a problem that we leave open.


\begin{thebibliography}{XXX}



\bibitem[Ba]{ba} Ballmann, W., Nonpositively curved manifolds of higher rank. {\it Ann. of Math. (2)} \textbf{122} (1985), no. 3, 597-609.




\bibitem[BuSp]{busp} Burns, K.; Spatzier, R., Manifolds of nonpositive curvature and their buildings. {\it Inst. Hautes \'Etudes Sci. Publ. Math.} \textbf{65} (1987), 35-59.

\bibitem[Con]{con} Connell, C., A characterization of hyperbolic rank one negatively curved homogeneous spaces. {\it Geom. Dedicata} \textbf{128} (2002), 221-246.

\bibitem[Co]{co} Constantine, D., $2$-frame flow dynamics and hyperbolic rank-rigidity in nonpositive curvature. {\it J. Mod. Dyn.} \textbf{2} (2008), no. 4, 719-740.



\bibitem[GuZh]{guzh} Guler, D.; Zheng, F., Nonpositively curved compact Riemannian manifolds with degenerate Ricci tensor. {\it Transactions of the A.M.S.} \textbf{363} (2010), no. 3, 1265-1285.

\bibitem[Ha]{ha} Hamenst\"adt, U., A geometric characterization of negatively curved locally symmetric spaces. {\it J. Differential Geom.} \textbf{34} (1991), no. 1, 193-221.



\bibitem[Ko1]{ko1} Kowalski, O., An explicit classification of  $3$-dimensional Riemannian spaces satisfying $R(X,Y)\cdot R=0$. {\textit Czechoslovak Math. J.} \textbf{46(121)} (1996), no. 3, 427-474.

\bibitem[Ko2]{ko2} Kowalski, O., A classification of Riemannian $3$-manifolds with constant principal Ricci curvatures $\rho_1=\rho_2\neq\rho_3$. {\textit Nagoya Math. J.} \textbf{132} (1993), 1-36.

\bibitem[Mi]{mi} Milnor, J.,  Curvatures of left invariant metrics on Lie groups. {\it Advances in Math.} \textbf{21} (1976), no. 3, 293-329.

\bibitem[ScWo]{scwo} Schmidt, B.; Wolfson, J., Complete curvature homogenous metrics on $SL_2(\R)$.  Preprint.

\bibitem[Se]{se} Sekigawa, K., On some $3$-dimensional curvature homogeneous spaces. {\it Tensor (N.S.)} \textbf{31} (1977), no. 1, 87-97.

\bibitem[ShSpWi]{shspwi} Shankar, K.; Spatzier, R.; Wilking, B., Spherical rank rigidity and Blaschke manifolds. {\it Duke Math. Journal}, \textbf{128} (2005), 65-81.

\bibitem[Si]{si} Singer, I. M., Infinitesimally homogeneous spaces. {\it Comm. Pure Appl. Math.} \textbf{13} (1960), 685-697.


\end{thebibliography}
\end{document}